                    \def\version{July 26, 2016}                   %

 \documentclass[reqno,11pt]{amsart}
 \usepackage{amsmath, amsthm, a4, latexsym, amssymb}
\usepackage[unicode]{hyperref}
\usepackage{srcltx}
\usepackage{xcolor}

\def\@rmrk#1#2{\refstepcounter
    {#1}\@ifnextchar[{\@yrmrk{#1}{#2}}{\@xrmrk{#1}{#2}}}

\makeatletter\@addtoreset{equation}{section}\makeatother

 \newfont{\bfit}{cmbxti10 scaled 1200}

\renewcommand{\d}{{\rm d}}
 \newcommand{\e}{{\rm e} }

 \newcommand{\eps}{\varepsilon}

 \newcommand{\R}{\mathbb{R}}
 \newcommand{\N}{\mathbb{N}}
 \newcommand{\Z}{\mathbb{Z}}

 \newcommand{\E}{\mathbb{E}}
 \renewcommand{\P}{\mathbb{P}}
 \renewcommand{\thefootnote}{\arabic{footnote}}
 \def\1{{\mathchoice {1\mskip-4mu\mathrm l} 
{1\mskip-4mu\mathrm l}
{1\mskip-4.5mu\mathrm l} {1\mskip-5mu\mathrm l}}}

 \newcommand{\Mcal}{{\mathcal M}}

\newcommand{\heap}[2]{\genfrac{}{}{0pt}{}{#1}{#2}}
\newcommand{\sfrac}[2]{\mbox{$\frac{#1}{#2}$}}

\renewcommand{\subsection}{\secdef \subsct\sbsect}
\newcommand{\subsct}[2][default]{\refstepcounter{subsection}
\vspace{0.15cm}
{\flushleft\bf \arabic{section}.\arabic{subsection}~\bf #1  }
\nopagebreak\nopagebreak}
\newcommand{\sbsect}[1]{\vspace{0.1cm}\noindent
{\bf #1}\vspace{0.1cm}}

{\nopagebreak {\hfill\rule{2mm}{2mm}}\\ }

\newtheorem{theorem}{Theorem}[section]
\newtheorem{lemma}[theorem]{Lemma}
\newtheorem{cor}[theorem]{Corollary}
\newtheorem{prop}[theorem]{Proposition}

\newtheoremstyle{thm}{1.5ex}{1.5ex}{\itshape\rmfamily}{}
{\bfseries\rmfamily}{}{2ex}{}

\newtheoremstyle{rem}{1.3ex}{1.3ex}{\rmfamily}{}
{\itshape\rmfamily}{}{1.5ex}{}
\theoremstyle{rem}
\newtheorem{remark}{{\slshape\sffamily Remark}}[]

\refstepcounter{subsubsection}

\def\thebibliography#1{\section*{References}
  \list%
  {\arabic{enumi}.}
    {\settowidth\labelwidth{[#1]}\leftmargin\labelwidth
    \advance\leftmargin\labelsep
    \parsep0pt\itemsep0pt
    \usecounter{enumi}}
    \def\newblock{\hskip .11em plus .33em minus .07em}
    \sloppy                   
    \sfcode`\.=1000\relax}

\begin{document}
\title[Mean-field interaction of Brownian occupation measures, I]
{\large Mean-field interaction of\\ Brownian occupation measures, I:\\ uniform tube property of the Coulomb functional}
\author[Wolfgang K\"onig and Chiranjib Mukherjee ]{}
\maketitle
\thispagestyle{empty}
\vspace{-0.5cm}

\centerline{Wolfgang K\"onig\footnote{WIAS Berlin, Mohrenstra{\ss}e 39, Berlin 10117, {\tt koenig@wias-berlin.de},  and TU Berlin} and Chiranjib Mukherjee\footnote{Courant Institute, New York, 251 Mercer Street, New York 10012, USA, {\tt mukherjee@cims.nyu.edu},  and WIAS Berlin}}
\renewcommand{\thefootnote}{}

\vspace{-0.5cm}
\centerline{\textit{WIAS Berlin and TU Berlin, Courant Institute New York and WIAS Berlin}}
\vspace{0.2cm}

\begin{center}
\version
\end{center}

\begin{quote}{\small {\bf Abstract: }
We study the transformed path measure arising from the self-interaction of a three-dimensional Brownian motion via an exponential tilt with the Coulomb energy of the occupation measures of the motion by time $t$. The logarithmic asymptotics of the partition function were identified in the 1980s by Donsker and Varadhan \cite{DV83-P} in terms of a variational formula. Recently \cite{MV14} a new technique for studying the path measure itself was introduced, which allows for proving that the normalized occupation measure asymptotically concentrates around the set of all maximizers of the formula. In the present paper, we show that likewise the Coulomb functional of the occupation measure concentrates around the set of corresponding Coulomb functionals of the maximizers in the uniform topology. This is a decisive step on the way to a rigorous proof of the convergence of the normalized occupation measures towards an explicit mixture of the maximizers, derived in \cite{BKM15}. Our methods rely on 
deriving H{\"o}lder-continuity of the Coulomb functional of the occupation measure with exponentially small deviation probabilities and invoking the large deviation theory developed in \cite{MV14} to a certain shift-invariant functional of the occupation measures.
}
\end{quote}


\section{Introduction and main results}\label{intro}

\noindent In this paper, we study a transformed path measure that arises from a mean-field type interaction of a three dimensional Brownian motion in a Coulomb potential. Under the influence of such a transformed measure, the large-$t$ behavior of the normalized occupation measures, denoted by $L_t$, is of high interest. This is intimately connected to the well-known polaron problem from statistical mechanics and a full understanding of the behavior of $L_t$ under the aforementioned transformation is crucial for the analysis of the polaron path measure under \lq strong coupling\rq\,, its effective mass and justification of mean-field approximations. For physical relevance of this model, we refer to \cite{S86}. Some mathematically rigorous research in this direction began in the 1980s with the analysis of the partition function of Donsker and Varadhan (\cite{DV83-P}), but it was not until recently that a new technique was developed \cite{MV14} for handling the actual path measures, and the main results of
the present paper, besides being interesting on their own, make determinant contribution towards a 
deeper analysis and a full identification of the limiting distribution of $L_t$ under the transformed path measure. 

We start with developing the mathematical layout of the model in Section~\ref{sec-model}, remind on earlier results in Section~\ref{sec-earlier}, present our new progress in Section~\ref{sec-results}
and report on the achievements of \cite{MV14} in Section~\ref{compactLDP}, which plays an important role in the present context. 

\subsection{The transformed path measure.} \label{sec-model}

\noindent We start with the Wiener measure $\P$ on $\Omega=C([0,\infty),\R^3)$ corresponding to a $3$-dimensional Brownian motion $W=(W_t)_{t\geq 0}$ starting from the origin. We are interested in the transformed path measure
\begin{equation}\label{Phat}
\widehat{\P}_t(\d \omega)=\frac 1 { Z_t}\, \exp\bigg\{\frac 1t\int_0^t\int_0^t \d \sigma\d s \,\frac 1{\big|\omega_\sigma-\omega_s\big|}\bigg\} \, \P(\d \omega) \qquad \omega\in \Omega,
\end{equation}
with the normalizing constant, the {\it partition function},
\begin{equation}\label{Zhat}
Z_t= \E\bigg[ \exp\bigg\{\frac 1t\int_0^t\int_0^t \d \sigma\d s \,\frac 1{\big|W_\sigma-W_s\big|}\bigg\}\bigg].
\end{equation}
We remark that the asymptotic behavior of $\widehat\P_t$ is determined by those influential paths which make $|W_\sigma- W_s|$ small, i.e., the interaction is {\it{self-attractive}}.  We also remark  that the factor $\frac 1t$ in the exponent in \eqref{Phat} makes the model interesting. Indeed, the double-integral in the exponent is of order $t^2$ for paths that stay in a compact region, and the entropic cost for this behaviour is $\e^{-O(t)}$; it is relatively easy to suspect that such a behaviour is typical under the transformed measure. Hence, it is the factor $\frac1t$ that makes the energy and the entropy terms run on the same scale and still gives the path enough freedom to fluctuate.

Let  
\begin{equation}\label{OcMeas}
L_t = \frac 1t \int_0^t  \d s \, \delta_{W_s}
\end{equation}
be the normalized occupation measure of $W$ until time $t$. This is a random element of $\Mcal_1(\R^3)$, the space of probability measures on $\R^3$. Then the path measure $\widehat{\P}_t$ can be written as 
$$
\widehat{\P}_t(A)=\frac 1{Z_t}\, \E\big[\1_A\, \exp\big\{t H(L_t)\big\}\big] \qquad A\subset \Omega,
$$
where  
\begin{equation}\label{Hdef}
H(\mu)= \int_{\R^3}\int_{\R^3}  \frac {\mu(\d x)\,\mu(\d y)}{|x-y|},\qquad \mu\in\Mcal_1(\R^3),
\end{equation}
denotes the {\it{Coulomb potential energy functional}} of $\mu$. Hence, $ \widehat \P_t$ is an exponential tilt of the Coulomb energy function of $L_t$ with parameter $t$. It is the goal of this paper to make a contribution to a rigorous  understanding of the behavior of $L_t$ under $\widehat {\P}_t$.

For any $\mu\in\Mcal_1(\R^3)$, we define the function
$$
\big(\Lambda\mu\big)(x)= \bigg(\mu\star \frac 1{|\cdot|}\bigg)(x)= \int_{\R^3} \frac{\mu(\d y)}{|x-y|},
$$
which is also sometimes called its {\it{Coulomb potential energy functional}}. In order to avoid misunderstandings, we will call $H(\mu)$ the {\em Coulomb energy} and $\Lambda(\mu)$ the {\em Coulomb functional} of $\mu$. Note that $H(\mu)= \big\langle \mu, \Lambda \mu\big\rangle= \int (\Lambda\mu)(x) \,\mu(\d x)$. We remark that the Coulomb functional of the Brownian occupation measure,  
\begin{equation}\label{Lambdadef}
\Lambda_t(x)=\big(\Lambda L_t\big)(x)= \int_{\R^3} \frac{L_t(\d y)}{|x-y|}= \frac 1 t \int_0^t \frac{\d s}{|W_s- x|},
\end{equation}
is almost surely finite in $\R^3$.

\subsection{Existing results.}\label{sec-earlier}

\noindent Donsker and Varadhan \cite{DV83-P} studied the asymptotic behaviour of  $Z_{t}$ resulting in the variational formula
 \begin{equation}\label{rhodef}
 \begin{aligned}
\lim_{t\to\infty}\frac 1t\log Z_t &= \sup_{\mu\in \Mcal_1(\R^3)} \bigg\{ H(\mu)- I(\mu)\bigg\}\\
&=\sup_{\heap{\psi\in H^1(\R^3)}{\|\psi\|_2=1}} 
\Bigg\{\int_{\R^3}\int_{\R^3}\d x\d y\,\frac {\psi^2(x) \psi^2(y)}{|x-y|} -\frac 12\big\|\nabla \psi\big\|_2^2\Bigg\} 
=\rho,
\end{aligned}
\end{equation}
with $H^1(\R^3)$ denoting the usual Sobolev space of square integrable
functions with square integrable gradient.  Furthermore, we put
\begin{equation}\label{Idef}
I(\mu)=\frac 12\|\nabla\psi\|_2^2
\end{equation}
if $\mu$ has a density $\psi^2$ with $\psi\in H^1(\R^3)$, and $I(\mu)=\infty$ otherwise. Note that both $H$ and $I$ are shift-invariant functionals, i.e., $H(\mu)=H(\mu\star\delta_x)$  and $I(\mu)=I(\mu\star\delta_x)$ for any $x\in\R^3$. 

The above result is a consequence of a {\it{large deviation principle}} (LDP) for $L_t$ under $\P$ in $\Mcal_1(\R^3)$, developed by Donsker and Varadhan (\cite{DV75}). This means, when $\Mcal_1(\R^3)$
is equipped with the usual weak topology, for every open set $G\subset\Mcal_1(\R^3)$, 
\begin{equation}\label{dvlb}
\liminf_{t\to\infty}\frac 1t \log \P\big(L_t\in G\big)\geq - \inf_{\mu\in G} I(\mu),
\end{equation}
and for any compact set $K\subset \Mcal_1(\R^3)$, 
\begin{equation}\label{dvub}
\limsup_{t\to\infty}\frac 1t \log \P\big(L_t\in K\big)\leq - \inf_{\mu\in K} I(\mu).
\end{equation}
The above statement is also called a {\it{weak large deviation principle}} since the upper bound \eqref{dvub} holds only for compact subsets. We say that a family of probability distributions satisfies a {\it{strong large deviation principle}} if, along with the lower bound \eqref{dvlb}, the upper bound \eqref{dvub} holds also for all closed sets. 
\medskip\noindent

The variational formula \eqref{rhodef} has been analyzed by Lieb (\cite{L76}). It turns out that there is a smooth, rotationally symmetric and centered maximizer $\psi_0$ which is  unique except for spatial translations. In other words, if $\mathfrak m$ denotes the set of maximizing densities, then
\begin{equation}\label{shiftunique}
\mathfrak m=  \big\{\mu_0 \star \delta_x\colon x\in \R^3\big\},
\end{equation}
where $\mu_0$ is a probability measure with a density $\psi_0^2$ so that $\psi_0$ maximizes the variational problem \eqref{rhodef}. We will often write $\mu_x=\mu_0\star \delta_x$ and write $\psi_x^2$ for its density.
\medskip\noindent

 

Given \eqref{rhodef} and \eqref{shiftunique}, we expect the distribution of $L_t$ under the transformed measure $\widehat \P_t$ to concentrate around $\mathfrak m$ and, even more, to converge towards 
a mixture of spatial shifts of $\mu_0$. Such a precise analysis was carried out by Bolthausen and Schmock \cite{BS97} for a spatially discrete version of $\widehat\P_t$, i.e., for the continuous-time simple random walk on $\Z^d$ instead of Brownian motion and an interaction potential $v\colon \Z^d \rightarrow [0,\infty)$ with finite support instead of the singular Coulomb potential $x\mapsto 1/|x|$. A first key step in \cite{BS97} was to show that, under the transformed measure, the probability of the local times falling outside any neighborhood of the maximizers decays exponentially. For its proof, the lack of a strong LDP for the local times was handled by an extended version of a standard periodization procedure by folding the random walk into some large torus. Combined with this, an explicit tightness property of the distributions of the local times led to an identification of the limiting distribution.

However, in the context of the continuous setting with a singular Coulomb interaction, the aforementioned periodization technique or any standard compactification procedure does not work well to circumvent the lack of a strong LDP. An investigation of $\widehat\P_t \circ L_t^{-1}$, the distribution of  $L_t$ under $\widehat \P_t$, remained open until a recent result \cite{MV14} rigorously justified the above heuristics, leading to the statement 
\begin{equation}\label{tubeweak}
\limsup_{t\to\infty} \frac 1t \log \widehat\P_t \big\{L_t \notin U(\mathfrak m)\big\}<0,
\end{equation}
where $U(\mathfrak m)$ is any neighborhood of $\mathfrak m$ in the weak topology induced by the Prohorov metric, the metric that is induced by all the integrals against continuous bounded test functions. Hence, \eqref{tubeweak} implies that the distribution of $L_t$ under $\widehat \P_t$ is asymptotically concentrated around $\mathfrak m$. Since a one-dimensional picture of $\mathfrak m$ is an infinite line, its neighborhood resembles an infinite tube. Therefore, assertions similar to \eqref{tubeweak} are sometimes called a {\em tube property}. 

It is worth pointing out that, although \eqref{tubeweak} requires only the weak topology in the statement, its proof is crucially based on a robust theory of a {\it{compactification}} $\widetilde {\mathcal X}$ of the quotient space
$$
\widetilde\Mcal_1(\R^d) \hookrightarrow \widetilde {\mathcal X}
$$ 
of orbits $\widetilde \mu=\{\mu \star\delta_x\colon x\in\R^d\}$ of probability measures $\mu$ on $\R^d$ (for any $d\geq 1$) under translations and a full LDP for the distributions of $\widetilde L_t\in \widetilde\Mcal_1(\R^d)$ embedded in the compactification. In particular, this is based on a topology induced by a different metric in the compactification $\widetilde {\mathcal X}$, see Section \ref{compactLDP} for details and its consequences in the present context.

\subsection{Our results: uniform tube property and regularity of $\boldsymbol{\Lambda(L_t)}$}\label{sec-results}

\noindent  Let us turn to our main results. We write 
$$
\Lambda(\psi^2)(x)=\int \d y\,\frac{\psi^2(y)}{|x-y|}
$$ 
for functions $\psi^2$, and recall that $\psi_w^2= \psi_0^2 \star \delta_w$ denotes the shift of the maximizer $\psi_0^2$ of the second variational formula \eqref{rhodef} by $w\in\R^3$. 
Roughly speaking, we will establish that on the large deviations scale, the Coulomb functionals $\Lambda(L_t)$ 
under the transformed path measures $\widehat\P_t$ stay close to the manifold of Coulomb functionals $\Lambda\mathfrak m=\{\Lambda\psi_w^2\colon w\in \R^3\}$ 
acting on the translations of the Pekar maximizers, see Theorem \ref{tubeunifmetric}. This is a first determinant step towards establishing the full conjecture on the convergence 
of the distributions $\widehat\P_t\circ L_t^{-1}$ towards an explicit spatial mixture of the maximizers $\mathfrak m$, see Remark \ref{remark2}.
On the way towards proving Theorem \ref{tubeunifmetric}, we also derive some modulus of continuity of $\Lambda(L_t)$, which can be of independent interest
in the realm of regularity properties of local times for stochastic processes.


 Here is the statement of our first main result.

\begin{theorem}\label{tubeunifmetric}
For any $\eps>0$, 
\begin{equation}\label{eqtubeunifmetric}
\limsup_{t\to\infty}\frac 1t\log \widehat\P_t\bigg\{\inf_{w\in \R^3} \big\|\Lambda_t- \Lambda \psi_w^2\big\|_\infty > \eps\bigg\} <0.
\end{equation}
\end{theorem}

This is a tube property for $\Lambda_t$ in the uniform metric, since the $\eps$-neighbourhood of $\Lambda(\mathfrak m)=\{\Lambda(\psi^2_w)\colon w\in\R^3\}$ can be visualized as a tube around the \lq line\rq\ $\mathfrak m$. The proof of Theorem~\ref{tubeunifmetric} is given in Section \ref{sectiontubeunifmetric}. 

As a consequence of Theorem \ref{tubeunifmetric}, the Hamiltonian $H(L_t)=\langle L_t,\Lambda L_t\rangle$ converges in distribution towards the common Coulomb energy of any member of $\mathfrak m$ and we state this fact as

\begin{cor}\label{convergeHLt}
Under $\widehat\P_t$, the distributions of $H(L_t)$ converge weakly to the Dirac measure at
$$
H(\psi_0^2)= \int\int_{\R^3\times\R^3} \frac{\psi_0^2(x)\psi_0^2(y)}{|x-y|}\, \d x\d y.
$$
\end{cor}

Let us highlight the core of the proof of Theorem \ref{tubeunifmetric}. An important technical hindrance in the proof of Theorem~\ref{tubeunifmetric} stems from the singularity of the Coulomb potential $x\mapsto1/|x|$, which does not fit within the set up
of standard large deviation theory. This problem was encountered also in \cite{MV14} for deriving \eqref{tubeweak}. As it concerns $L_t$, this turned out to be a mild technical issue. Indeed,  a simple truncation argument with replacing $1/|x|$ by its regularized version $1/\sqrt{|x|^2+\delta^2}$ sufficed to carry over the
theory developed in \cite{MV14} to this singular potential. However, as we need now to work with $\Lambda(L_t)$ in the {\em uniform} metric, the singularity of $1/|\cdot|$ turns out be a more serious problem, since a standard contraction principle combined with the truncation argument does not work well here. Instead, we need a strategy that shows a strong regularity property of the random map $x\mapsto \Lambda_t(x)$, more precisely, an exponential decay of the probability that its modulus of continuity deviates from zero.  This is our second main result.

\begin{theorem}\label{superexp}
For every $b>0$, 
\begin{equation}\label{superexpeq}
\lim_{\delta\to 0}\limsup_{t\to\infty}\frac 1t \log\, \P\bigg\{\sup_{x_1,x_2\in\R^3\colon |x_1-x_2|\leq \delta} \, \big| \Lambda_t(x_1)-\Lambda_t(x_2)\big| \geq b\bigg\}=-\infty.
\end{equation}
\end{theorem}
In Section \ref{sectionsuperexp}, we prove Theorem~\ref{superexp}.
Let us state the following useful corollary to Theorem \ref{superexp}, which is also of independent interest; its proof is also deferred to Section~\ref{sectiontubeunifmetric}.

\begin{cor}\label{expdecaysupnorm}
 For any $b>0$,
 $$
 \limsup_{t\to\infty} \frac 1t \log \P\big\{\|\Lambda_t\|_\infty>b\big\}<0.
 $$
 \end{cor}
 
Concerning the regularity of $\Lambda_t$ we have a more quantitative result than Theorem~\ref{superexp}, which we state here because of its own interest. Indeed, one main step in the proof of Theorem~\ref{superexp} is the following (stretched) exponential integrability.

\begin{prop}\label{expmomLambda}
There are constants $\rho>1$, $a\in(0,1)$ and $\beta\in(0,\infty)$ such that
$$
\sup_{\heap{x_1,x_2\in\R^3}{|x_1-x_2|\leq 1}} \sup_{x\in \R^3}\E_x \bigg[\exp\bigg\{\beta \bigg(\frac{\big|\Lambda_1(x_1)-\Lambda_1(x_2)\big|}{|x_1-x_2|^{a}}\bigg)^\rho\bigg\}\bigg] <\infty.
$$
\end{prop}

This assertion suffices for our purposes, but it is clear that our proof can be extended to prove a number of more refined statements about the regularity of $\Lambda_t$, like the identification of the exact index of its H\"older continuity, almost sure limsup and liminf assertions about its modulus of continuity, and local and global laws of iterated logarithms. Let us remark that 
this might run prallel to the work of Donsker and Varadhan \cite{DV77} on the law of iterated logarithm for one-dimensional Brownian local times.

\begin{remark}\label{remark2}
Let us remark that Theorem~\ref{tubeunifmetric}, in combination with Corollaries \ref{convergeHLt} and  \ref{expdecaysupnorm}, besides their intrinsic interests on their own right,
have proved to be instrumental in proving tightness of the distributions of $L_t$ under $\widehat\P_t$ and their convergence 
towards an explicit (spatially inhomogeneous) mixture of the maximizers $\{\psi_x^2\colon  x\in \R^3\}$, which 
resolves the aforementioned ``mean-field approximation" of the polaron problem on the level of path measures.
This has been carried out in \cite{BKM15}, and in this context, we refer to Section 2.4 in \cite{BKM15} for a heuristic discussion on the relevance of the results derived in the present paper.
\end{remark}

\begin{remark}\label{remark3}
Let us also remark that our choice $d=3$ and the Coulomb potential $V(x)=1/|x|$ in $\R^3$ is motivated by the {\it{polaron problem}} in quantum statistical mechanics, 
which is described by the path measures 
$$
\widehat\P_{\lambda,t}(\d \omega)= \frac 1 {Z_{\lambda,t}} \exp\bigg\{\lambda \int_0^t\int_0^t \d\sigma \d s \frac{\e^{-\lambda|\sigma-s|}}{|\omega_\sigma- \omega_s|}\bigg\} \,\, \P(\d \omega),
$$
where $\lambda>0$ is a parameter, $\P$ is three-dimensional Wiener measure and $Z_{\lambda,t}$ is the normalization constant or the partition function. The quest for a rigorous analysis
of the polaron measure $\widehat\P_{\lambda,t}$ for $t\to\infty$, followed by $\lambda\to 0$, motivated and inspired the study of the mean-field path measures $\widehat\P_t$. These measures now have been fully analyzed
in the series of papers \cite{MV14}, \cite{M15}, \cite{BKM15} and the present article. See Section 1.1 in \cite{BKM15} for a discussion regarding the approximation of $\widehat\P_{\lambda,t}$ by $\widehat\P_t$ for $\lambda\sim 0$.
Let us emphasize that, in each of these three papers, an instrumental r\^{o}le is played by the uniqueness (modulo spatial shifts) of the maximizer of the variational formula in \eqref{rhodef} proved in \cite{L76}. Such a statement is known to hold only for $d=3$ and the Coulomb potential $V(x)=1/|x|$. Given such a uniqueness statement for another rotationally symmetric potential (possibly carrying a different singularity) in $\R^d$ with $d\geq 4$, it is conceivable that the methods of the present article will allow a similar analysis for the corresponding mean-field measures $\widehat\P_t$. However, due to its deep rooted connection to the polaron problem in mathematical physics, we content ourselves only with the present case pertaining to $d=3$ and  $V(x)=1/|x|$.
\end{remark}

\subsection{Review: compactness and large deviations}\label{compactLDP}

\noindent We now turn to the second main ingredient for the proof of Theorem \ref{tubeunifmetric}, which is based on the results derived in \cite{MV14}.
Since this will play an important role in our proof, we take the opportunity to introduce the main idea in \cite{MV14} and review its salient assertions.

Note that the space $\Mcal_1(\R^d)$ of probability measures in $\R^d$ fails to be compact in the weak topology, which is due to several reasons. For instance, 
the location of the mass can shift away to $\infty$ as for the sequence $(\mu\star\delta_{a_n})_n$ with $a_n\to\infty$, 
or the mass can be spread thinly and totally disintegrate into dust, like for a sequence of Gaussians with diverging variance.
Similarly, a mixture like $\mu_n=\frac{1}{2}[\mu\ast\delta_{a_n}+\mu\ast\delta_{-a_n}]$ can split into two (or more) widely separated pieces
if $a_n\to\infty$. To compactify this space one should be allowed to ``center" each such piece separately, as well as to allow some mass to be ``thinly spread and disappear". 
Let
$$
\widetilde\Mcal_1(\R^d)=\{\widetilde\mu\colon \mu\in \Mcal_1(\R^d)\}
$$ 
denote the quotient space of orbits $\widetilde \mu=\{\mu\star \delta_x\colon x\in\R^d\}$ of $\Mcal_1(\R^d)$ under translations.
Then intuitively, for any sequence $(\widetilde\mu_n)_n$ in $\widetilde\Mcal_1(\R^d)$ in the limit, one imagines, an empty, finite or countable collection $\{\alpha_j\colon \, j\in J\}$ 
of sub-probability distributions  that are widely separated with total mass $\sum_{j\in J} \alpha_j (\R^d)=p\le1$ and the remaining mass $1-p$ having totally disintegrated. For example, let $\mu_n$ be a mixture of three Gaussians, one with mean $0$ and variance $1$, one with mean $n$ and variance $1$ and one with mean $0$ and variance $n$, each with equal weight $\frac 13$. Then the limiting object is the collection $\{\widetilde\alpha_1,\widetilde\alpha_1\}$, where $\widetilde\alpha_1$ is the equivalence class of a Gaussian with variance $1$ and weight $\frac 13$.

This intuition naturally inspires the introduction of the space 
$$
\widetilde{\mathcal X}=\Big\{\xi=(\widetilde\alpha_j)_{j\in J}\colon J\mbox{ at most countable}, \alpha_j \in \Mcal_{\leq 1}(\R^d)\,\forall j\in J\Big\}
$$ 
of empty, finite or countable collections of orbits $\{{\widetilde \alpha}_j\colon \, j\in J\}$ of sub-probability distributions $\alpha_j$ having masses $p_j$ with $p=\sum_j p_j\le 1$.
Note that we have a canonical embedding
$$
\widetilde \Mcal_1(\R^d)\hookrightarrow \widetilde{\mathcal X}.
$$
In the proof of Theorem \ref{tubeunifmetric}, the following results will play an important role.

\begin{theorem}[\cite{MV14}, Theorem 3.2]\label{thm1MV14}
There is a metric $\mathbf D$ on $\widetilde {\mathcal X}$ so that $\widetilde\Mcal_1(\R^d)$ is dense in $(\widetilde {\mathcal X}, \mathbf D)$ and any sequence $(\widetilde\mu_n)_n$ in $\widetilde\Mcal_1(\R^d)$ finds a subsequence which converges in the metric $\mathbf D$ to some element $\xi\in \widetilde {\mathcal X}$. In other words, $\widetilde {\mathcal X}$ is the compactification of $\widetilde\Mcal_1(\R^d)$ and also the completion under the metric $\mathbf D$ of the totally bounded space $\widetilde\Mcal_1(\R^d)$.
\end{theorem}

\begin{theorem}[\cite{MV14}, Theorem 4.1]\label{thm2MV14}
 The distribution of the orbits $\widetilde L_t$ of the Brownian occupation measures embedded in the compact metric space $(\widetilde {\mathcal X}, \mathbf D)$ satisfy a strong LDP with the rate function 
$$
\widetilde{ J}(\xi)= \sum_{j\in J} \widetilde I(\widetilde\alpha_j) =\sum_{j\in J} I(\alpha_j),\qquad \xi=(\widetilde\alpha_j)_{j\in J}\in \widetilde {\mathcal X},
$$
where we recall that $I(\cdot)$ is defined in \eqref{Idef} and is shift-invariant and for any $\alpha\in\Mcal_{\leq 1}(\R^d)$, $I(\alpha)$ is a function only of the orbit $\widetilde\alpha$, which we call $\widetilde I(\widetilde \alpha)$. 
\end{theorem}

Let us now choose $d=3$ and recall the transformed path measure $\widehat\P_t$ from \eqref{Phat}.

 \begin{theorem}[\cite{MV14}, Theorem 5.3]\label{thm3MV14}
 The family of distributions of $\widetilde L_t $ under $\widehat\P_t$  satisfies a strong LDP in $\widetilde{\mathcal X}$ with rate function
 $$
{\widehat J}(\xi)= \widehat\rho-\sum_j \bigg\{\int_{\R^3}\int_{\R^3} \frac{1}{|x-y|} \alpha_j(\d x)\alpha_j(\d y)-{\widetilde I}(\widetilde{\alpha}_j)\bigg\}, \qquad\xi=\{\widetilde{\alpha}_j\}\in \widetilde{\mathcal X},
$$
and  $\widehat\rho$ is given by
\begin{equation}\label{rhohat}
\widehat\rho=\sup_{\xi\in\widetilde{\mathcal X}} \sum_j \bigg\{\int_{\R^3}\int_{\R^3} \frac{\psi_j^2(x)\psi_j^2(y)}{|x-y|} \d x\d y-\frac{1}{2}\sum_j\big\|\nabla \psi_j\big\|_2^2\bigg\}
\end{equation}
 and $\alpha_j(\d x)=\psi^2_j(x)\d x$ with $\sum_j \int_{\R^3} \psi_j^2(x) \d x\leq 1$.
\end{theorem}
Let us finally remark that the above theory applies to any shift-invariant functional $f$ of $L_t$, since $f(L_t)=\widetilde f(\widetilde L_t)$ for an obviously defined lifting $\widetilde f$ of $f$ to the space of orbits. 
For example, in Theorem \ref{thm3MV14} the theory was applied to $H(L_t)$, recall \eqref{Hdef}.
In the present paper, such shift-invariant dependence of $\|\Lambda_t\|_\infty$ on $L_t$ is exhibited by the simple identity
$$
\|\Lambda_t\|_\infty= \sup_{y\in\R^3} \bigg(\int_{\R^3} \frac{L_t(\d z)}{|z-y|}\bigg)= \sup_{y\in\R^3} \bigg(\int_{\R^3} \frac{\big(L_t\star \delta_x\big)(\d z)}{|z-y|}\bigg)= \|\Lambda_t\star\delta_x\|_\infty \quad\forall x\in\R^3,
$$
and is of crucial importance in the context of deriving Theorem \ref{tubeunifmetric} from the above theory, see the proof of \eqref{unifclaim4} in Section \ref{sectiontubeunifmetric}.

\section{Super-exponential estimate: Proof of Theorem \ref{superexp}}\label{sectionsuperexp}

\noindent  For any $x\in \R^3$ we will denote by $\P_x$ the Wiener measure for
the Brownian motion $W=(W_t)_{t\geq 0}$ starting at $x$ and by $\E_x$ the corresponding expectation and we continue to write $\P_0=\P$ and $\E_0=\E$. 
First we turn to the proof of Proposition~\ref{expmomLambda}, which follows from the following lemma.

\begin{lemma}\label{lemma1}
For any $\eps\in (0,1/2)$, if $a=1-2\eps$ and $\rho= \frac 1{1-\eps}$, then, for some $\beta\in (0,\infty)$,
\begin{equation}\label{eqlemma1}
\sup_{\heap{x_1,x_2\in\R^3}{|x_1-x_2|\leq 1}} \sup_{x\in \R^3}\E_x \bigg[\exp\bigg\{\beta \bigg(\frac{\big|\Lambda_1(x_1)-\Lambda_1(x_2)\big|}{|x_1-x_2|^{a}}\bigg)^\rho\bigg\}\bigg] <\infty.
\end{equation}
\end{lemma}
\begin{proof}
We fix $x_1, x_2\in \R^3$ with $|x_1-x_2|\leq 1$ and denote
$$
V(y)=V_{x_1,x_2}(y)= \frac 1{|y-x_1|}- \frac 1{|y-x_2|}, \qquad y \in \R^3,
$$
so that $\Lambda_1(x_1)-\Lambda_1(x_2)= \int_0^1 V(W_s) \, \d s$. Then by Jensen's inequality, 
$$
\bigg|\int_0^1V(W_s)\,\d s \bigg|^\rho\leq \int_0^1|V(W_s)|^\rho\,\d s.
$$
Let us now recall Khas'minski's lemma (see \cite[p.8]{S98}, \cite{P76}), which states that, if for a function $\widetilde V\ge 0$,
$$
\sup_{x\in \R^d} \E_{ x}\bigg\{\int_0^1 \widetilde V(W_s)\d s\bigg\}\le \eta<1,
$$
then
$$
\sup_{x\in \R^d} \E_{ x}\bigg\{\exp\bigg\{\int_0^1 \widetilde V(W_s)\d s\bigg\}\bigg\}\le \frac{\eta}{1-\eta} <\infty.
$$
Hence, \eqref{eqlemma1} follows for some $\beta\in (0,\infty)$ if we show that
\begin{equation}\label{eq2lemma1}
\sup_{\heap{x_1,x_2\in\R^3\colon}{|x_1-x_2|\leq 1}}|x_1-x_2|^{-a\rho} \sup_{x\in \R^3} \E_x\bigg[\int_0^1 | V(W_s)|^\rho\,\d s\biggr]<\infty.
\end{equation}
Now let us introduce a generic constant $C$  that does not depend on $x, x_1, x_2, y$, nor on any integration variable, and may change its value from line to line. 

We estimate, for any $x_1,x_2$ satisfying $|x_1-x_2|\leq 1$, and $a=1-2\eps$,
\begin{equation}\label{Vhatesti}
\begin{aligned}
 |V(y)|=\frac{\big| |y-x_2|-|y-x_1|\big|}{|y-x_1|\,|y-x_2|} &\leq \frac{|x_1-x_2|}{|y-x_1|\,|y-x_2|}\\
 &\leq |x_1-x_2|^a\,\, \frac{ \big[|y-x_2|^{1-a}+|y-x_1|^{1-a}\big]}{|y-x_1|\,|y-x_2|}.
 \end{aligned}
\end{equation}
The latter inequality follows from $(r+s)^{1-a}\leq r^{1-a}+s^{1-a}$ for any $r,s\geq0$. Furthermore, let us estimate the integral 
$$
h(y)=h_x(y)=\int_0^1 \d t\, \frac{\e^{-|x-y|^2/2t}}{t^{3/2}}
$$
as follows. Since for any $b>0$, the map $[1,\infty)\ni z\mapsto z^{3/2-b} \e^{-z}$ is bounded, we can estimate
$$
\begin{aligned}
\int_0^{|y-x|^2\wedge 1}\d t\,\frac{\e^{-|y-x|^2/2t}}{t^{3/2}} 
&\leq C|y-x|^{-3-2b}\int_0^{|y-x|^2\wedge 1}\d t\,\e^{-|y-x|^2/2t} \Big(\frac{ |y-x|^2}{2t}\Big)^{3/2+b} t^b\\
&\leq C |y-x|^{-3-2b}\int_0^{|y-x|^2\wedge 1}\d t\,t^b\\
&\leq C |y-x|^{-3-2b}\big(|y-x|^2\wedge 1\big)^{1+b},\qquad x,y\in\R^3.
\end{aligned}
$$
For the remaining integral, we have the upper bound
$$
\int_{|y-x|^2\wedge 1}^1 \d t\,\frac{\e^{-|y-x|^2/2t}}{t^{3/2}} \leq\int_{|y-x|^2\wedge 1}^1 \d t\,t^{-3/2}
\leq \big[|y-x|^2\wedge 1\big]^{-1/2}-1.
$$
Combining the preceding two estimates, we obtain that
\begin{equation}\label{trankernelesti}
h(y)=\int_{0}^1 \d t\,\frac{\e^{-|y-x|^2/2t}}{t^{3/2}} \leq C \frac 1{|y-x|(1+|y-x|)^b}.
\end{equation}
Let us now combine \eqref{Vhatesti} and \eqref{trankernelesti}, to get
$$
\begin{aligned}
|x_1-x_2|^{-a\rho} \,\, \E_x\bigg[\int_0^1 |V(W_s)|^\rho\,\d s\bigg]
&=(2\pi)^{-3/2}\int_{\R^3}\d y\,|x_1-x_2|^{-a\rho} |\widehat V(y)|^\rho \, h(y)\\
&\leq C\int_{\R^3}\d y\,\frac{ |y-x_2|^{\rho(1-a)}+|y-x_1|^{\rho(1-a)}}{|y-x_1|^\rho\,|y-x_2|^\rho}\,\frac 1{|y-x|(1+|y-x|)^b}.
\end{aligned}
$$
Taking the symmetry in $x_1$ and $x_2$ into account, we see that \eqref{eq2lemma1} follows once we have
$$
\sup_{\heap{x_1,x_2\in\R^3\colon}{|x_1-x_2|\leq 1}}\sup_{x\in\R^3}\int_{\R^3}\,\frac {\d y}{(1+|y-x|)^b}\,\frac1{|y-x_1|^\rho}\times\frac 1{|y-x|}\times\frac1{|y-x_2|^{\rho a}}<\infty.
$$
For this, we apply H\"older's inequality to the measure $\frac {\d y}{(1+|y-x|)^b}$ and the other three functions with parameters $p_1,p_2,p_3>1$ satisfying $\frac 1{p_1} + \frac 1{p_2} +\frac 1{p_3}=1$. Hence, it suffices to show that all the integrals
$$
\int_{\R^3}\,\frac {\d y}{(1+|y-x|)^b}\,\frac 1{|y-x_1|^{\rho p_1}},\qquad \int_{\R^3}\,\frac {\d y}{(1+|y-x|)^b}\, \frac 1{|y-x|^{p_2}},\qquad \int_{\R^3}\,\frac {\d y}{(1+|y-x|)^b}\, \frac 1{|y-x_2|^{\rho a p_3}},
\qquad 
$$
are bounded in $x,x_1,x_2$ for proper choices of $p_1,p_2,p_3$ and $b$. But this is ensured by requiring $b>3$ and $p_1<3/\rho$ and $p_2=\rho p_1$ (enforcing that $p_3=p_1\rho/(p_1\rho-\rho+1)$) and $p_3<3/a\rho$. The latter mean that $\frac {3(\rho-1)}{\rho(3-\rho a)}<p_1<\frac 3\rho$ and are possible as soon as $4>\rho(1+a)$. But this is satisfied for our choices $\rho=\frac 1{1-\eps}$ and $a=1-2\eps$, for any $\eps\in(0,1)$.

This finishes the proof of Lemma \ref{lemma1}.
\end{proof}

\begin{lemma}\label{expmomLambdafullspace}
Fix $\eps\in (\frac 13,\frac 12)$ and choose $a=1-2\eps$ and $\rho=\frac 1 {1-\eps}$ as in Lemma \ref{lemma1}. Then there exists a constant $\beta_1=\beta_1(\eps)>0$ such that the random variable
\begin{equation}\label{Mdef}
M=\int_{\R^3}\d x_1 \int_{\R^3}\d x_2\,\1\{|x_1- x_2|\leq 1\} \,\, \bigg[\exp\bigg\{\beta_1 \bigg(\frac{\big|\Lambda_1(x_1)-\Lambda_1(x_2)\big|}{|x_1-x_2|^{a}}\bigg)^\rho\bigg\}-1\bigg]
\end{equation}
has a finite expectation under $\P_0$.
\end{lemma}

\begin{proof}
By Lemma \ref{lemma1} and Fubini's theorem, it suffices to show that 
\begin{equation}\label{claim2}
\int\int_{|x_1- x_2|\leq 1} \d x_1 \d x_2 \,\, \E\bigg[\exp\bigg\{\beta_1 \bigg(\frac{\big|\Lambda_1(x_1)-\Lambda_1(x_2)\big|}{|x_1-x_2|^{a}}\bigg)^\rho\bigg\}-1\bigg] < \infty.
\end{equation}

We decompose $\R^3\subset \bigcup_{n=0}^\infty\big\{x\in\R^3\colon n\leq |x|< n+1\big\}$ and put $\tau_n= \inf\{t>0\colon\, |W_t| > n- n^\alpha\}$ for some $\alpha\in (0,1)$. 
For any $r>0$ and $x\in\R^3$, we also denote by $B_r(x)$ the open Euclidean ball of radius $r$ around $x$.
Then 
\begin{equation}\label{taubigless1}
\begin{aligned}
&\int\int_{|x_1- x_2|\leq 1} \d x_1 \d x_2 \,\, \E\bigg[\exp\bigg\{\beta_1 \bigg(\frac{\big|\Lambda_1(x_1)-\Lambda_1(x_2)\big|}{|x_1-x_2|^{a}}\bigg)^\rho\bigg\}-1\bigg] \\
&\leq \sum_{n=0}^\infty \int_{|x_1|\in [n,n+1)} \d x_1 \int_{B_1(x_1)} \d x_2 \,\, \bigg[ \E\bigg\{\1_{\{\tau_n>1\}} \bigg(\exp\bigg\{\beta_1 \bigg(\frac{\big|\Lambda_1(x_1)-\Lambda_1(x_2)\big|}{|x_1-x_2|^{a}}\bigg)^\rho\bigg\}-1\bigg)\bigg\}
\\
&\qquad\qquad\qquad\qquad\qquad+\E\bigg\{\1_{\{\tau_n\leq1\}} \bigg(\exp\bigg\{\beta_1 \bigg(\frac{\big|\Lambda_1(x_1)-\Lambda_1(x_2)\big|}{|x_1-x_2|^{a}}\bigg)^\rho\bigg\}-1\bigg\}\bigg)\bigg].
\end{aligned}
\end{equation}
The first expectation inside the integrals is handled as follows. We note that, with $|x_1| \in [n,n+1)$  and $x_2\in B_1(x_1)$, if $\tau_n >1$, then $|W_s- x_1| > n^\alpha$ and $|W_s- x_2| > n^\alpha - 1$ for any $s\in[0,1]$. Hence, for any $n\in\N$, on the event $\{\tau_n>1\}$,
$$ 
\frac{\big|\Lambda_1(x_1)-\Lambda_1(x_2)\big|}{|x_1-x_2|^{a}} \leq \frac {|x_1-x_2|}{|x_1-x_2|^{1-2\eps}} \, \int_0^1 \frac {\d s}{|W_s -x_1| |W_s- x_2|}\leq c_1 |x_1-x_2|^{2\eps} n^{-2\alpha}\leq c_1  n^{-2\alpha}.
$$
Hence, 
\begin{equation}\label{taubig1}
\begin{aligned}
&\sum_{n=0}^\infty \int_{|x_1|\in [n,n+1)} \d x_1 \int_{B_1(x_1)} \d x_2 \,\, \E\bigg\{\1_{\{\tau_n>1\}} \bigg(\exp\bigg\{\beta_1 \bigg(\frac{\big|\Lambda_1(x_1)-\Lambda_1(x_2)\big|}{|x_1-x_2|^{a}}\bigg)^\rho\bigg\}-1\bigg)\bigg\} \\
&\leq \sum_{n=0}^\infty \bigg(\e^{\beta_1 c_1^\rho n^{-2\alpha\rho}} - 1\bigg) \mbox{Leb}\bigg\{x_1\in \R^3\colon\,|x_1|\in [n,n+1)\bigg\}\mbox{Leb}(B_1(0)).
\end{aligned}
\end{equation}
Since the first term is of size $O(n^{-2\alpha\rho})$ and the first Lebesgue measure is of size $O(n^2)$, the above sum is finite for $\alpha> \frac 3{2\rho}$. Since we chose $\eps>\frac 13$ and hence $\rho= \frac 1 {1-\eps} > \frac 32$, we can choose some $\alpha\in (0,1)$ so that $\alpha > \frac 3 {2\rho}$, as desired. 

Let us now handle the second expectation in \eqref{taubigless1}. By the Cauchy-Schwarz inequality and Proposition \ref{expmomLambda}, if $\beta_1$ is small enough, for any $x_1,x_2\in\R^3$ such that $|x_1-x_2|\leq1$,
$$
\begin{aligned}
\E\bigg[\1_{\{\tau_n\leq1\}} &\bigg\{\exp\bigg\{\beta_1 \bigg(\frac{\big|\Lambda_1(x_1)-\Lambda_1(x_2)\big|}{|x_1-x_2|^{a}}\bigg)^\rho\bigg\}-1\bigg\}\bigg] \\
&\leq \P\big(\tau_n\leq 1\big)^{\frac 12} \,\, \E\bigg[\exp\bigg\{2\beta_1 \bigg(\frac{\big|\Lambda_1(x_1)-\Lambda_1(x_2)\big|}{|x_1-x_2|^{a}}\bigg)^\rho\bigg\}\bigg]^{\frac 12} \\
&\leq C \P\bigg(\max_{[0,1]} W > n- n^\alpha\bigg)^{\frac 12},
\end{aligned}
$$
where $C$ does not depend on $x_1,x_2$. Since the last probability is of order $\e^{-c n^2}$, the second sum on $n$ in \eqref{taubigless1} is obviously finite. This, combined with the finiteness of the sum in \eqref{taubig1}, proves \eqref{claim2} and hence  finishes the proof of  Lemma~\ref{expmomLambdafullspace}.
\end{proof}

For the proof of Theorem \ref{superexp} we will use the following (multidimensional) estimate of Garsia-Rodemich-Rumsey \cite[p.\,60]{SV79}.

\begin{lemma}\label{GRR}
Let $p(\cdot)$ and $\Psi(\cdot)$ be strictly increasing continuous functions on $[0,\infty)$ so that $p(0)=\Psi(0)=0$ and $\lim_{t\uparrow\infty} \Psi(t)=\infty$. If $f\colon \R^d\to \R$ is continuous on the closure of the ball $B_{2r}(z)$ for some $z\in \R^d$ and $r>0$, then the bound 
\begin{equation}\label{GRR1}
\int_{B_r(z)} \d x \int_{B_r(z)}  \d y\,\, \Psi\bigg(\frac{|f(x)-f(y)|}{p(|x-y|)}\bigg) \leq M<\infty,
\end{equation}
implies that
\begin{equation}\label{GRR2}
\big|f(x)- f(y)\big| \leq 8 \int_0^{2|x-y|} \Psi^{-1}\bigg(\frac{M}{\gamma u^{2d}}\bigg) \, p(\d u), \qquad x,y\in B_r(z),
\end{equation}
for some constant $\gamma$ that depends only on $d$.
\end{lemma}

Finally we are ready to prove Theorem \ref{superexp}. 

{\bf{Proof of Theorem \ref{superexp}.}} 
The Brownian scaling property implies that
$$
\Lambda_t(x)= \frac 1t \int_0^t \frac 1{|W_s-x|}\, \d s 
= \int_0^1 \frac 1 {|W(ts)-x|} \,\d s 
\stackrel{{\mathcal D}}{=} \int_0^1 \frac 1 {|\sqrt t W(s)- x|}\, \d s 
=t^{-1/2} \Lambda_1(x t^{-1/2}),
$$
where $\stackrel{{\mathcal D}}{=}$ denotes equality in distribution. Hence, the claim of Theorem \ref{superexp} is equivalent to
\begin{equation}\label{claim5}
\lim_{\delta\to 0}\limsup_{t\to\infty}\frac 1t \log \P\bigg\{ \sup_{x_1,x_2\in\R^3\colon |x_1- x_2|\leq \delta t^{-1/2}} \, \big| \Lambda_1(x_1)- \Lambda_1(x_2)\big| \geq b t^{1/2}\bigg\}=-\infty, \qquad b>0.
\end{equation}

Now we would like to apply Lemma \ref{GRR}. We pick $\eps\in (\frac 13, \frac 12)$ and $a=1-2\eps$ and $\rho= \frac 1 {1-\eps}$ and $\beta=\beta_1$ as in Lemma \ref{expmomLambdafullspace} and choose
\begin{equation}\label{choices}
\Psi(x)= \e^{\beta |x|^\rho}-1, \qquad p(x)= |x|^{a }= |x|^{1-2\eps}, \qquad f(x)=\Lambda_1(x).
\end{equation}
Then $\Psi(\cdot)$, $p(\cdot)$ and $f(\cdot)$ all satisfy the requirements of Lemma \ref{GRR}. Furthermore, Lemma \ref{expmomLambdafullspace} implies that hypothesis \eqref{GRR1} is satisfied if $|x_1- x_2|\leq \delta$ and $\delta>0$ is chosen small enough, where the random variable $M$ is given in \eqref{Mdef}.
Hence, \eqref{GRR2} implies that for $|x_1-x_2|\leq\delta t^{-1/2}$ and all $t\geq 1$, 
\begin{equation}\label{GRR3}
\big|\Lambda_1(x_1)-\Lambda_1(x_2)\big| \leq 8 \int_0^{\delta t^{-1/2}} \Psi^{-1}\bigg(\frac M {\gamma u^{6}}\bigg) \,p(\d u)
=8\frac {1-2\eps}{\beta^{1/\rho}} \int_0^{\delta t^{-1/2}} \,  \log\bigg(1+\frac{M}{\gamma u^6}\bigg)^{1/\rho} \, u^{-2\eps} \, \d u.
\end{equation}
For $u\in(0,\delta t^{-1/2}]$ and all sufficiently large $t$, we estimate
$$
8\frac {1-2\eps}{\beta^{1/\rho}}\log\bigg(1+\frac{M}{\gamma u^6}\bigg)^{1/\rho} \leq C \Big((\log (M\vee 1))^{1/\rho}+(\log \sfrac 1u)^{1/\rho}\Big),
$$
for some constant $C$ that does not depend on $t$ if $t$ is sufficiently large. Hence, the right-hand side of \eqref{GRR3} is not larger than
$$
C_\delta(\log (M\vee 1))^{1/\rho}t^{\eps-1/2}+C_\delta (\log t)^c t^{\eps-1/2}
$$
for some $C_\delta, c$, not depending on $t$, and $C_\delta\to 0$ as $\delta\to 0$. Substituting this in \eqref{GRR3} and recalling that $\rho=\frac 1 {1-\eps}$, we obtain
\begin{equation}\label{claim6}
\begin{aligned}
\P\bigg\{ &\sup_{x_1,x_2\in\R^3\colon |x_1- x_2|\leq \delta t^{-1/2}} \, \big| \Lambda_1(x_1)- \Lambda_1(x_2)\big| \geq bt^{1/2}\bigg\} \leq \P\bigg\{(\log (M\vee 1))^{1/\rho} +(\log t)^c\geq \frac{b}{C_\delta} t^{1-\eps}\bigg\}\\
&\leq\P\bigg\{ \log (M\vee 1)\geq \,\frac{b^\rho}{C_\delta^{\rho}}\, t-C_2(\log t)^{c\rho})\bigg\}
\leq \E(M\vee 1)\exp\bigg\{-\frac{b^\rho}{C_\delta^{\rho}}\, t+C_2(\log t)^{c\rho}\bigg\}.
\end{aligned}
\end{equation}
Recall that by Lemma \ref{expmomLambdafullspace}, $\E(M\vee 1)<\infty$. If we now let $t\to\infty$, followed by $\delta\to 0$, the above estimate now implies \eqref{claim5} and therefore Theorem \ref{superexp}.
\qed

\begin{cor}\label{P_tsuperexp}
For any $b>0$, 
$$
\lim_{\delta\to0}\limsup_{t\to\infty}\frac 1t \log\widehat\P_t\bigg\{\sup_{x_1,x_2\in\R^3\colon |x_1- x_2|\leq \delta} \, \big| \Lambda_t(x_1)- \Lambda_t(x_2)\big| \geq b\bigg\} =-\infty.
$$
\end{cor}
\begin{proof}
Let us denote by $A_{t,\delta}$ the above event inside the probability. Then the Cauchy-Schwarz inequality gives that
$$
\frac 1t \log \widehat\P_t\big\{A_{t,\delta}\big\}\leq \frac 1{2 t} \log \E\big\{\e^{2t H(L_t)}\big\}- \frac 1t \log \E\big\{\e^{t H(L_t)}\big\} + \frac 1 2 \frac 1t \log \P\big\{A_{t,\delta}\big\}.
$$ 
While the first two terms have finite large-$t$ limits,  by Theorem \ref{superexp} the large-$t$ limit of the third term tends to $-\infty$ as $\delta\to0$. This proves the corollary.
\end{proof}

\section{LDP for $\Lambda_t$ in the uniform metric: Proof of Theorem \ref{tubeunifmetric}}\label{sectiontubeunifmetric}
Recall that we need to show, for any $\eps>0$, 
\begin{equation}\label{unifclaim1}
\limsup_{t\to\infty}\frac 1t \log \widehat \P_t \bigg\{\inf_{w\in \R^3} \big\|\Lambda_t- \Lambda\psi_w^2\big\|_\infty \geq \eps\bigg\} <0. 
\end{equation}
We approximate the sup-norm inside the probability via  a coarse graining argument as follows. For any $\delta\in(0,1)$, we can estimate 
\begin{equation}\label{unifest1}
\begin{aligned}
\inf_{w\in \R^3} \big\|\Lambda_t- \Lambda\psi_w^2\big\|_\infty
&= \inf_{w\in \R^3} \sup_{x\in \R^3}\big|\Lambda_t(x)- \big(\Lambda\psi_w^2)(x)\big| \\
&\leq \sup_{x_1,x_2\in\R^3\colon |x_1-x_2|\leq \delta} \big|\Lambda_t(x_1)-\Lambda_t(x_2)\big| \\
&\qquad+ \, \inf_{w\in\R^3} \sup_{z\in \delta \Z^3} \bigg[ \big|\Lambda_t(z)-\big(\Lambda\psi_w^2)(z)\big| + \, \sup_{\tilde z \in B_\delta(z)} \big| \big(\Lambda\psi_w^2\big)(\tilde z)-\big(\Lambda\psi_w^2\big)(z)\big|\bigg].
\end{aligned}
\end{equation}
Note that, for any $w\in \R^3$ the deterministic function $\Lambda\psi_w^2$ is uniformly continuous on $\R^3$ and hence 
$$
\lim_{\delta\downarrow 0} \sup_{z\in \delta \Z^3}\sup_{\tilde z \in B_\delta(z)} \big| \big(\Lambda\psi_w^2\big)(\tilde z)-\big(\Lambda\psi_w^2\big)(z)\big|=0.
$$
Since $\eps>0$ is arbitrary, the above fact and Corollary \ref{P_tsuperexp} imply that, to deduce \eqref{unifclaim1}, it suffices to prove, for any $\eps,\delta>0$,  
\begin{equation}\label{unifclaim2}
\limsup_{t\to\infty} \frac 1t \log \widehat\P_t\bigg\{ \inf_{w\in\R^3}\, \sup_{z\in \delta\Z^3} \big| \Lambda_t(z)-\big(\Lambda\psi_w^2)(z)\big| \geq \eps\bigg\}<0.
\end{equation}
For any $z\in \delta\Z^3$, $w\in \R^3$ and any $\eta>0$, we can estimate
\begin{equation}\label{unifest2}
\big| \Lambda_t(z)-\big(\Lambda\psi_w^2)(z)\big| 
\leq \int_{B_\eta(z)} \frac{\psi_w^2(y)}{|y-z|}\, \d y + \int_{B_\eta(z)} \frac{L_t(\d y)}{|y-z|} + \bigg| \int_{\R^3} \frac{\1\{|y-z|\geq \eta\}}{|y-z|} \bigg(L_t(\d y)- \psi_w^2(y)\d y\bigg)\bigg|.
\end{equation}
The first term can be handled easily. Note that, for any $w\in \R^3$, $\psi_w$ is radially symmetric and $\|\psi_w\|_2=1$. Hence using polar coordinates and invoking the dominated convergence theorem we can argue 
that 
\begin{equation}\label{unifestdet}
\lim_{\eta\to 0} \sup_{z\in \delta \Z^3}  \int_{B_\eta(z)} \frac{\psi_w^2(y)}{|y-z|} \d y=0.
\end{equation}
Let us turn to the second term in \eqref{unifest2}. We claim that, for any $\delta>0$ and $\eta>0$ small enough, 
\begin{equation}\label{unifclaim3}
 \limsup_{t\to\infty} \frac 1t \log \widehat\P_t\bigg\{ \sup_{z\in \delta\Z^3} \int_{B_\eta(z)}\frac{L_t(\d y)}{|y-z|}\geq \eps \bigg\}<0.
\end{equation}
Let us first handle the above event with the Wiener measure $\P$ replacing $\widehat\P_t$.  Then we can estimate
\begin{equation}\label{unifest3}
\P\bigg\{ \sup_{z\in \delta\Z^3} \int_{B_\eta(z)}\frac{L_t(\d y)}{|y-z|}>\eps \bigg\} \leq \sum_{\heap{z\in \delta\Z^3}{|z|\leq t^2}}  \P\bigg\{\int_{B_\eta(z)}\frac{L_t(\d y)}{|y-z|}\geq \eps/2 \bigg\} 
\,\,+\P\bigg\{ \sup_{\heap{z\in \delta\Z^3}{|z|> t^2}} \int_{B_\eta(z)}\frac{L_t(\d y)}{|y-z|}\geq \eps/2 \bigg\}.
\end{equation}
The second term can be estimated by the probability that the Brownian path, starting at origin, travels a distance $t^2-\eps$ by time $t$. This probability is of order $\exp\{-c t^3\}$ and can be ignored. 
For the first term we note that a box of size $t^2$ in $\R^3$ can be covered by $O(t^6)$ sub-boxes of side length $\delta$ and that the probability is maximal for $z=0$. Hence, we can estimate, with the help of Markov's inequality,
for any $\beta>0$,
\begin{equation}\label{unifest4}
\sum_{\heap{z\in \delta\Z^3}{|z|\leq t^2}}  \P\bigg\{\int_{B_\eta(z)}\frac{L_t(\d y)}{|y-z|}>\eps/2 \bigg\} 
\leq Ct^6 \, \P\bigg\{\beta\int_0^t V_\eta(W_s) \d s>t \beta\eps/2  \bigg\}
\leq C t^6 \e^{-\frac \eps 2 t \beta} \,\,\E\bigg\{\e^{\beta\int_0^t V_\eta(W_s) \,\d s}\bigg\},
\end{equation}
where $V_\eta(x) =\1_{\{|x|\leq \eta\}} \, \frac 1 {|x|}$. Note that, for any $\beta>0$ and some constants $c_1,c_2$ independent of $\eta$,
$$
\sup_{y\in \R^3} \E_y \bigg\{\beta\int_0^{1} V_\eta(W_s)\, \d s\bigg\} \leq \beta \int_{B_\eta(0)} \frac {\d x}{|x|} \int_0^{1} p_s(0,x) \,\d s \leq  \beta c_1\int_{B_\eta(0)} \frac {\d x}{|x|^2}= c_2 \eta \beta.
$$
For any fixed $\beta>0$ and $\eta$ small enough, this is not larger than $1/2$, and by Khas'minskii's lemma (\cite[p.~8]{S98}), successive conditioning and the Markov property,
$$
\E\bigg\{\e^{\beta\int_0^t V_\eta(W_s) \d s}\bigg\} \leq 2^{\lceil t \rceil }.
$$
Then \eqref{unifest4} and \eqref{unifest3} imply, for any $\beta>0$, 
$$
 \limsup_{t\to\infty} \frac 1t \log \P\bigg\{ \sup_{z\in \delta\Z^3} \int_{B_\eta(z)}\frac{L_t(\d y)}{|y-z|}\geq \eps \bigg\} \leq -\eps\beta/2 + \log 2.
$$
From this we can deduce \eqref{unifclaim3} by choosing $\beta>0$ large enough and invoking H\"older's inequality as in the proof of Corollary \ref{P_tsuperexp}. We
drop the details to avoid repetition. 

Let us turn to the third term on the right hand side of \eqref{unifest2}. Then by \eqref{unifestdet} and \eqref{unifclaim3}, it suffices to prove that, for every $\eta, \eps>0$,
\begin{equation}\label{unifclaim4}
\limsup_{t\to\infty} \frac 1t \log \widehat\P_t\big\{L_t \in F_\eta\big\} < 0,
\end{equation}
where
$$
F_\eta=\bigg\{ \mu\in \Mcal_1(\R^3)\colon \,\,\forall w\in \R^3\,\,\sup_{z\in \R^3} \big|\big\langle f_{z,\eta}, \mu- \psi_w^2\big\rangle \big| \geq \eps \bigg\},
$$
where we put $f_{z,\eta}(y)= \frac 1{|y-z|}\wedge \frac 1\eta$. We claim that for each $\eta>0$, $F_\eta$ is a closed set in the weak topology in $\Mcal_1(\R^3)$. First note that the family $\mathcal A_\eta= \{f_{z,\eta}\colon z \in \R^d\}$ is equicontinuous and uniformly bounded. Hence, for any $\eta>0$, the set
$$
G_{\eta,w}= \bigg\{\mu\in\Mcal_1(\R^3)\colon  \sup_{f\in\mathcal A_\eta} \big|\big\langle f, \mu- \psi_w^2\big\rangle \big| < \eps\bigg\}
$$
is weakly open and hence 
$$
F_\eta= \bigcap_{w\in \R^d} \,\,G_{\eta,w}^{\rm c}
$$
is weakly closed. Furthermore, we note that $F_\eta$ is shift-invariant, i.e., if $\mu\in F_\eta$, then $\mu\star \delta_x\in F_\eta$ for any $x\in\R^3$. In other words, 
$$
 \widehat\P_t\big\{L_t \in F_\eta\big\}=\widehat\P_t\big\{\widetilde L_t \in \widetilde F_\eta\big\},
 $$
where $\widetilde F_\eta=\{\widetilde \mu\colon \mu\in F_\eta\}$, the set of orbits $\widetilde\mu=\{\mu\star\delta_x\colon x\in \R^3\}$ of members of $F_\eta$, is a closed set in $\widetilde\Mcal_1(\R^3) \hookrightarrow \widetilde{\mathcal X}$, and by Theorem \ref{thm1MV14}, $\widetilde F_\eta$ is also compact in $(\widetilde{\mathcal X},\mathbf D)$. Then by Theorem \ref{thm3MV14},
 $$
 \limsup_{t\to\infty}\frac 1t \log\widehat\P_t\big\{\widetilde L_t \in \widetilde F_\eta\big\} \leq - \inf_{\xi\in \widetilde F_\eta} \widehat J(\xi).
 $$
According to  \cite[Lemma 5.4]{MV14}, the variational formula in \eqref{rhohat} attains its maximum only in trivial sequences $\xi$ consisting of just one single orbit of a 
probability measure $\mu(\d x)=\psi^2(x)\,\d x$ with $\psi$ a rotationally symmetric, $L^2$-normalized function, which, by the uniqueness 
of the variational problem \eqref{rhodef} (recall \eqref{shiftunique}), must be one of the maximizers of the formula in \eqref{rhodef}, and $\rho=\widehat \rho$. Since $\widetilde F_\eta$ is in particular compact and does not contain such an element $\xi$, we have that $\inf_{\xi\in \widetilde F_\eta} \widehat J(\xi)>0$.  These two facts imply \eqref{unifclaim4} and hence Theorem \ref{tubeunifmetric}.
 \qed
 
 We end this section with the proof of Corollary \ref{expdecaysupnorm}.

{\it{Proof of Corollary \ref{expdecaysupnorm}.}}  The proof is straightforward and similar to the last line of arguments. Indeed, we note that for any $\delta>0$,
 $$
 \begin{aligned}
 \P\big\{\|\Lambda_t\|_\infty>b\big\}&\leq \P\bigg\{\sup_{|x_1-x_2|\leq \delta} \big|\Lambda_t(x_1) - \Lambda_t(x_2)\big| \geq b/2\bigg\}+ \P\big\{\sup_{x\in \delta \Z^3} \Lambda_t(x) \geq b/2\big\}\\ 
 & \leq  \P\bigg\{\sup_{|x_1-x_2|\leq \delta} \big|\Lambda_t(x_1) - \Lambda_t(x_2)\big| \geq b/2\bigg\}+ \P\bigg\{\sup_{\heap{x\in \delta \Z^3}{|x|\leq t^2}} \Lambda_t(x) \geq b/2\bigg\}\\
 &\qquad\qquad\qquad\qquad\qquad+\P\bigg\{\sup_{\heap{x\in \delta \Z^3}{|x|> t^2}} \Lambda_t(x) \geq b/2\bigg\} 
 \end{aligned}
 $$
By Theorem \ref{superexp}, the first term has a strictly negative exponential rate. The third term can again be neglected since this is of order $\exp\{-c t^3\}$. Also for the second term, the box of size $t^2$ can be covered by $O(t^6)$ sub-boxes of side length $\delta$. Therefore, 
 $$
\P\bigg\{\sup_{\heap{x\in \delta \Z^3}{|x|\leq t^2}} \Lambda_t(x) \geq b/2\bigg\}\leq  C t^6   \P\bigg\{\Lambda_t(0) >b/2\bigg\}.
 $$
 For any $\kappa >0$,
 $$
 \P\bigg\{\Lambda_t(0) >b/2\bigg\} \leq \e^{-\kappa bt/2} \, \E\bigg\{\exp\bigg\{\kappa\int_0^t \frac {\d s}{|W_s|}\bigg\}\bigg\}.
 $$
 We choose $t>u\gg 1$ and $\kappa>0$ small enough so that $\sqrt u \kappa\ll 1$ and
 $$
\alpha=\sup_{x\in \R^3} \E_x\bigg\{\kappa\int_0^u \frac{\d s} {|W_s|}\bigg\}= \E_0\bigg\{\kappa\int_0^u \frac{\d s} {|W_s|}\bigg\}=2 \kappa \sqrt u \E\bigg(\frac 1 {|W_1|}\bigg)\ll 1.
 $$ 
 Then by Khas'minskii's lemma \cite[p.~8]{S98},
 $$
 \sup_{x\in \R^3} \E_x\bigg\{\exp\bigg\{\kappa\int_0^u \frac {\d s} {|W_s|}\bigg\}\bigg\} \leq \frac 1 {1-\alpha},
 $$
 and by successive conditioning and the Markov property,
 $$
 \E\bigg\{\exp\bigg\{\kappa\int_0^t \frac {\d s} {|W_s|}\bigg\}\bigg\} \leq \bigg(\frac 1{1-\alpha}\bigg)^{t/u}.
 $$
 Since $\log(1+\alpha)\approx\alpha$ as $\alpha\to 0$, for any $b>0$ and $\kappa>0$ suitably chosen and $u$ large enough,
 $$
 \begin{aligned}
 \P\bigg\{\Lambda_t(0) >b/2\bigg\} &\leq \exp\bigg\{-\frac {\kappa b t}2 + \frac t u \log(1-\alpha)\bigg\}\\
 &\leq\exp\bigg[- t \kappa \bigg\{\frac b 2- \frac 1 {\sqrt u} c\bigg\}\bigg]
 \\
 &\leq\exp\big\{- t \kappa \widetilde C\big\}
  \end{aligned}
 $$
 for some $\widetilde C=\widetilde C(u,a, c)>0$. This proves the corollary.
\qed
 
\medskip 

{\bf{Acknowledgment.}} The second author would like to thank Erwin Bolthausen for suggesting this interesting problem and numerous useful discussions on the  model.

\end{document}